\documentclass[12pt]{amsart}

\usepackage{setspace}
\RequirePackage{amssymb}
\usepackage{fullpage}
\usepackage{amsbsy}
\usepackage{enumerate}
\usepackage{mathrsfs}
\usepackage{amsmath}
\usepackage{url}
\usepackage{xypic}

\theoremstyle{plain}
\newtheorem{thm}{Theorem}[section]
\newtheorem*{thm*}{Theorem}
\newtheorem*{mainthm*}{Main Theorem}
\newtheorem*{mainlem*}{Main Lemma}
\newtheorem{lem}[thm]{Lemma} \newtheorem*{lem*}{Lemma}
\newtheorem{claim}[thm]{Claim} \newtheorem*{claim*}{Claim}
 \newtheorem*{cor*}{Corollary}
 \newtheorem*{prop*}{Proposition}

\theoremstyle{definition}
\newtheorem{defn}[thm]{Definition} \newtheorem*{defn*}{Definition}

\theoremstyle{remark}
\newtheorem{rem}[thm]{Remark} \newtheorem*{rem*}{Remark}
 \newtheorem*{example*}{Example}
 \newtheorem*{conj*}{Conjecture}
 \newtheorem*{question*}{Question}

\newcommand{\Ord}{\mathrm{Ord}}

\begin{document}

\author{Trevor M.\ Wilson}
\title{A game-theoretic proof of Shelah's theorem on labeled trees}

\address{Department of Mathematics\\Miami University\\Oxford, Ohio 45056\\USA}
\email{twilson@miamioh.edu} 

\begin{abstract}
 We give a new proof of a theorem of Shelah which states that for every family of labeled trees, if the cardinality $\kappa$ of the family is much larger (in the sense of large cardinals) than the cardinality $\lambda$ of the set of labels, more precisely if the  partition relation $\kappa \to (\omega)^{\mathord{<}\omega}_\lambda$ holds, then there is a homomorphism from one labeled tree in the family to another. Our proof uses a characterization of such homomorphisms in terms of games.
\end{abstract}

\maketitle

\section{Introduction}

We work in ZFC. For every set $X$ we write $X^{\mathord{<}\omega}$ for the set of all finite sequences from $X$ and $[X]^{\mathord{<}\omega}$ for the set of all finite subsets of $X$.  For every natural number $n < \omega$ we write $X^n$ for the set of all $n$-length sequences from $X$ and $[X]^n$ for the set of all $n$-element subsets of $X$. When $X$ is a set of ordinals we may identify each element of $[X]^{\mathord{<}\omega}$ with its increasing enumeration, which is an element of $X^{\mathord{<}\omega}$.

For a nonzero cardinal $\lambda$ and an infinite cardinal $\kappa$, the partition relation $\kappa \to (\omega)^{\mathord{<}\omega}_\lambda$ means that for every function $f : [\kappa]^{\mathord{<}\omega} \to \lambda$ there is an infinite subset $H \subset \kappa$ that is \emph{homogeneous} for $f$, meaning that $f$ is constant on $[H]^n$ for all $n < \omega$. When $\lambda = 1$ this partition relation holds trivially for every infinite cardinal $\kappa$, but already when $\lambda = 2$, the least $\kappa$ such that $\kappa \to (\omega)^{\mathord{<}\omega}_\lambda$ is a large cardinal known as the Erd\H{o}s cardinal $\kappa(\omega)$.

A \emph{tree} on a set $X$ is a nonempty subset of $X^{\mathord{<}\omega}$ closed under initial segments, and the \emph{root} of such a tree is the empty sequence $\langle \rangle$.\footnote{For a more abstract approach, we could equivalently use rooted trees in the sense of graph theory.} A \emph{homomorphism} of trees is a function from one tree to another that preserves lengths of sequences and the initial segment relation. (Equivalently, it preserves the root and the predecessor relation.)

For a nonzero cardinal $\lambda$, a \emph{$\lambda$-labeled tree} (or \emph{$\lambda$-colored tree}) is a structure $\mathcal{T} = \langle T; \ell^{\mathcal{T}}\rangle$ where $T$ is a tree on some set $X$ and $\ell^\mathcal{T}$ is a function from $T$ to $\lambda$. A \emph{homomorphism} from a $\lambda$-labeled tree  $\mathcal{T} = \langle T; \ell^{\mathcal{T}}\rangle$ to a $\lambda$-labeled tree $\mathcal{U} = \langle U; \ell^{\mathcal{U}}\rangle$ is a homomorphism of trees $h : T \to U$ that preserves labels, meaning $\ell^{\mathcal{U}} \circ h = \ell^{\mathcal{T}}$.

We now state a theorem connecting partition relations to labeled trees. It follows from results of Shelah \cite{shelah1982better}; see Eklof and Shelah \cite[Theorem 2.1]{eklof1999absolutely} for a similar statement.
The converse also holds; Herden \cite{herden2012upper} is a good reference containing proofs of both directions.

\begin{thm}[Shelah]\label{thm:Shelah}
 Let $\kappa$ be an infinite cardinal and let $\lambda$ be a nonzero cardinal. If the partition relation $\kappa \to (\omega)^{\mathord{<}\omega}_\lambda$ holds, then for every sequence of $\lambda$-labeled trees $\langle \mathcal{T}_\alpha : \alpha < \kappa\rangle$ there is a homomorphism $\mathcal{T}_\alpha \to \mathcal{T}_\beta$ for some $\alpha < \beta < \kappa$.
\end{thm}

\begin{rem}
 In the simplest case $\lambda = 1$, this theorem reduces to the statement that for every sequence of (unlabeled) trees $\langle T_i : i < \omega\rangle$ there is a homomorphism $T_i \to T_j$ for some $i < j < \omega$. This is true because otherwise the ranks of the trees would form an infinite decreasing sequence in $\Ord \cup \{\infty\}$, which is a contradiction. Nevertheless, it may be interesting to observe that our game-theoretic proof works in this case also.
\end{rem}

Existing proofs of Theorem \ref{thm:Shelah} rely on the Nash-Williams theory of better-quasi-orderings.
We will give a short and simple proof avoiding this theory entirely, instead using the games defined in the next section.

\section{The game $G(\mathcal{T},\mathcal{U})$}

In this section we let $\lambda$ be a nonzero cardinal and let $\mathcal{T} = \langle T, \ell^\mathcal{T}\rangle$ and $\mathcal{U} = \langle U, \ell^\mathcal{U}\rangle$ be $\lambda$-labeled trees with $T \subset X^{\mathord{<}\omega}$ and  $U \subset Y^{\mathord{<}\omega}$ for some sets $X$ and $Y$. In the following game, we can think of the second player as continuously building finite partial homomorphisms from $\mathcal{T}$ to $\mathcal{U}$ in response to challenges from the first player.

\begin{defn}
 The game $G(\mathcal{T},\mathcal{U})$ is defined as follows. Players I and II alternately play elements of $X$ and $Y$ respectively for $\omega$ rounds:
 \begin{center}
  \begin{tabular}{c|cccccccc}
  I  & $x_0$ &       & $x_1$ &       & $x_2$ &       & $\ldots$ & \\
  II &       & $y_0$ &       & $y_1$ &       & $y_2$ &          & $\ldots$
 \end{tabular}
 \end{center}
 The players are subject to the following rules for all $n<\omega$:
 \begin{description}
  \item[Tree membership for player I] $\langle x_0,\ldots, x_n \rangle \in T$.
  \item[Tree membership for player II] $\langle y_0,\ldots, y_n \rangle \in U$.
  \item[Label-matching for player II] $\ell^{\mathcal{U}}(\langle y_0,\ldots, y_n \rangle) = \ell^{\mathcal{T}}(\langle x_0,\ldots, x_n \rangle)$.
 \end{description}
 The first player to break a rule loses. If both players follow the rules forever, then player II wins. If the roots of $\mathcal{T}$ and $\mathcal{U}$ have different labels, we say that player II loses immediately (this can be considered as the $n = -1$ case of the label-matching rule.)
\end{defn}

For notational convenience we allow play to continue after a rule is broken, even though it cannot affect the outcome. The fact that player II is considered the winner if both players follow the rules forever means that $G(\mathcal{T},\mathcal{U})$ is a \emph{closed game} for player II.

A \emph{winning strategy for player I} in $G(\mathcal{T},\mathcal{U})$ is a function $\Sigma : Y^{\mathord{<}\omega} \to X$ such that for every infinite sequence $\langle y_n : n < \omega\rangle$ of moves for player II, if $x_n = \Sigma(\langle y_0,\ldots,y_{n-1}\rangle)$ for all $n < \omega$, then player I wins.  Similarly (with a minor change to account for player I moving first) a \emph{winning strategy for player II} in $G(\mathcal{T},\mathcal{U})$ is a function $\Sigma : X^{\mathord{<}\omega} \setminus \{\langle \rangle\} \to Y$ such that for every infinite sequence $\langle x_n : n < \omega\rangle$ of moves for player I, if $y_n = \Sigma(\langle x_0,\ldots,x_n\rangle)$ for all $n < \omega$ then player II wins.

We will need the following simple observation relating winning strategies in this game to homomorphisms. (They are essentially just different notations for the same thing.)

\begin{lem}\label{lem:strategy-hom-equiv}
 There is a homomorphism from $\mathcal{T}$ to $\mathcal{U}$ if and only if player II has a winning strategy in the game $G(\mathcal{T},\mathcal{U})$.
\end{lem}
\begin{proof}
 If $h : \mathcal{T} \to \mathcal{U}$ is a homomorphism, then the strategy $\Sigma$ for player II in the game $G(\mathcal{T},\mathcal{U})$ given by $\Sigma(\langle x_0,\ldots,x_n \rangle) = y_n$ where $h(\langle x_0,\ldots,x_n \rangle) = \langle y_0,\ldots,y_n\rangle$ is clearly a winning strategy. Conversely, if $\Sigma$ is a winning strategy for player II in $G(\mathcal{T},\mathcal{U})$, then there is a homomorphism $h : \mathcal{T} \to \mathcal{U}$ given by
 \[h(\langle x_0,\ldots,x_{n-1}\rangle) = \big\langle \Sigma (\langle x_0 \rangle), \Sigma (\langle x_0,x_1\rangle), \ldots, \Sigma( \langle x_0,\ldots,x_{n-1} \rangle) \big \rangle. \qedhere\]
\end{proof}

\begin{rem}
 A consequence of Lemma \ref{lem:strategy-hom-equiv} that we will not use in this article, but is nevertheless worth pointing out, is that the existence of a homomorphism between two $\lambda$-labeled trees is absolute to any transitive model $\langle M; \in\rangle$ of ZFC containing both of them, by closed game absoluteness (see Kechris and Moschovakis \cite[Section 9B]{kechris1978notes}).
\end{rem}

We will use an immediate consequence of Lemma \ref{lem:strategy-hom-equiv} given by the Gale--Stewart theorem, which says that closed games are \emph{determined}, meaning that one player or the other (but clearly not both) must have a winning strategy. This gives a ``positive'' criterion for the nonexistence of a homomorphism:

\begin{lem}\label{lem:no-strategy-hom-equiv}
 There is no homomorphism from $\mathcal{T}$ to $\mathcal{U}$ if and only if player I has a winning strategy in the game $G(\mathcal{T},\mathcal{U})$.
\end{lem}

Besides Lemma \ref{lem:no-strategy-hom-equiv}, the only other ingredient in our proof of Shelah's theorem will be a method of combining several strategies for different games.  This method is often used to prove consequences of the axiom of determinacy such as the first periodicity theorem (see Moschovakis \cite[Diagram 6B.5]{moschovakis2009descriptive}.)

\section{Proof of Shelah's theorem}

Let $\kappa$ be an infinite cardinal, let $\lambda$ be a nonzero cardinal, let $\langle \mathcal{T}_\alpha : \alpha < \kappa\rangle$ be a sequence of $\lambda$-labeled trees, and assume $\kappa \to (\omega)^{\mathord{<}\omega}_\lambda$. Assume toward a contradiction that for all $\alpha < \beta < \kappa$ there is no homomorphism from $\mathcal{T}_\alpha$ to $\mathcal{T}_\beta$. Then by Lemma \ref{lem:no-strategy-hom-equiv} we may choose a winning strategy $\Sigma_{\alpha\beta}$ for player I in the game $G(\mathcal{T}_\alpha, \mathcal{T}_\beta)$ for all $\alpha < \beta < \kappa$.
  
We will combine these strategies to define a function $f : [\kappa]^{\mathord{<}\omega} \to \lambda$ as follows. Given an increasing finite sequence of ordinals $\langle \alpha_0,\ldots,\alpha_n \rangle \in [\kappa]^{\mathord{<}\omega}$, we can play the strategies $\Sigma_{\alpha_i\alpha_{i+1}}$ for $i < n$ against each other to produce an $n \times n$ triangular array of moves. This is shown for $n = 3$ in Figure \ref{fig:figure1}, where the dashed arrows indicate copied moves (feeding the output of one strategy into another) and the solid arrows indicate application of the chosen strategies for player I.  The initial moves for player I are also provided by the chosen strategies.
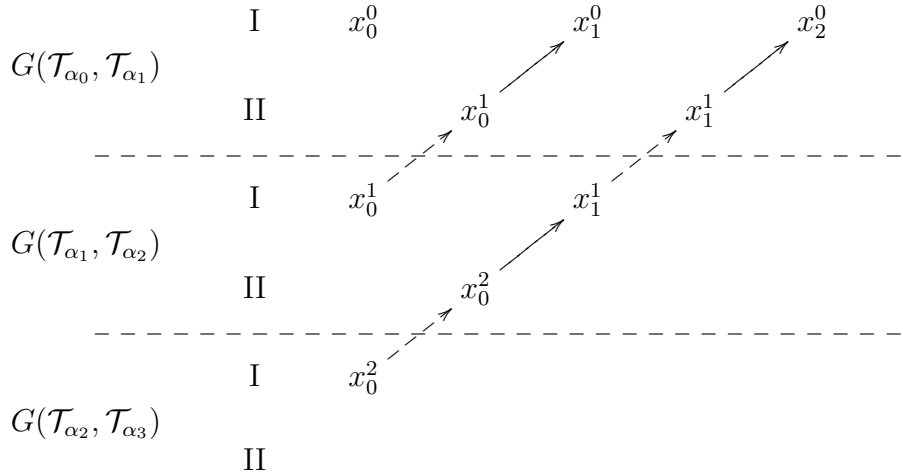
\begin{figure}[h]
 \begin{displaymath}
  \xymatrix@R-2.2pc{
                     & \text{I}  & x^0_0               &               & x^0_1               &               & x^0_2 & \\
   G(\mathcal{T}_{\alpha_0}, \mathcal{T}_{\alpha_1}) &&&               &                     &               &       & \\                                             
                     & \text{II} &                     & x^1_0\ar[uur] &                     & x^1_1\ar[uur] &       & \\
   \mbox{\vphantom{$G(\mathcal{T}_{\alpha_0})$}} \ar@{--}[rrrrrrr] &&& &                     &               &       & \\
                     & \text{I}  & x^1_0\ar@{-->}[uur] &               & x^1_1\ar@{-->}[uur] &               &       & \\
   G(\mathcal{T}_{\alpha_1}, \mathcal{T}_{\alpha_2}) &&&               &                     &               &       & \\                                               
                     & \text{II} &                     & x^2_0\ar[uur] &                     &               &       & \\
   \mbox{\vphantom{$G(\mathcal{T}_{\alpha_0})$}} \ar@{--}[rrrrrrr] &&& &                     &               &       & \\
                     & \text{I}  & x^2_0\ar@{-->}[uur] &               &                     &               &       & \\
   G(\mathcal{T}_{\alpha_2}, \mathcal{T}_{\alpha_3}) &&&               &                     &               &       & \\                                                
                     & \text{II} &                     &               &                     &               &       & \\                                               
  }
 \end{displaymath}
 \caption{Combining three strategies for player I}
 \label{fig:figure1}
\end{figure}

We call the sequence of ordinals  $\langle \alpha_0,\ldots, \alpha_n\rangle$ \emph{good} if all rules are followed in the resulting $n \times n$ triangular array.  In other words, player II has not yet lost, which implies that player I also has not yet lost because the array was generated by winning strategies for player I. As a trivial case, we consider every length-1 sequence $\langle \alpha \rangle$ to be good.

Note that a necessary and sufficient condition for the sequence $\langle \alpha_0,\ldots, \alpha_n\rangle$ to be good is that the roots of the trees $\mathcal{T}_{\alpha_0}, \ldots, \mathcal{T}_{\alpha_n}$ all have the same label and copying moves via the dashed arrows always satisfies the label-matching rule for player II.
(This condition is sufficient because each move is originally produced by a winning strategy for player I, and the tree membership rules are the same for both players.) 
 
If the sequence of ordinals $\langle \alpha_0,\ldots, \alpha_n\rangle$ is good then the sequence of moves $\langle x^0_0,\ldots, x^0_{n-1}\rangle$ obtained in the top row of the resulting $n \times n$ triangular array is a member of the tree $\mathcal{T}_{\alpha_0}$ and we may define  $f (\langle \alpha_0, \ldots, \alpha_n \rangle)$ as its label in that tree:
\begin{align*}
 f (\langle \alpha_0, \ldots, \alpha_n \rangle) = \ell^{\mathcal{T}_{\alpha_0}}(\langle x^0_0,\ldots, x^0_{n-1}\rangle).
\end{align*}
As a trivial case, $f (\langle \alpha \rangle)$ is the label of the root of $\mathcal{T}_{\alpha}$.
 
If the sequence $\langle \alpha_0,\ldots, \alpha_n\rangle$ is not good, we arbitrarily define $f(\langle \alpha_0,\ldots, \alpha_n\rangle) = 0$ in order to obtain a total function $f : [\kappa]^{\mathord{<}\omega} \to \lambda$. (These arbitrary values will not be used.)

The partition relation $\kappa \to (\omega)^{\mathord{<}\omega}_{\lambda}$ implies that some infinite subset $H \subset \kappa$ is homogeneous for $f$. Replacing $H$ by an initial segment of itself if necessary, we may assume that $H$ has order type $\omega$ and let $\langle \alpha_0,\alpha_1,\alpha_2,\ldots \rangle$ be its increasing enumeration. Homogeneity implies a shift-invariance property for finite intervals in $H$:
\begin{align}\label{eqn:weak-homog}
 f(\langle \alpha_i,\ldots,\alpha_{n+i-1} \rangle) = f(\langle \alpha_{i+1},\ldots,\alpha_{n+i}\rangle)
\end{align}
for all $i,n < \omega$.\footnote{The statement that for every $f : [\kappa]^{\mathord{<}\omega} \to \lambda$ there is an infinite increasing sequence of ordinals with this shift-invariance property implies Silver's ``weak'' partition relation $\kappa \overset{w}{\to} (\omega)^{\mathord{<}\omega}_\lambda$ and can easily be proved equivalent to it.  Silver \cite{silver1970large} proved that $\kappa \to (\omega)^{\mathord{<}\omega}_\lambda$ itself is equivalent to $\kappa \overset{w}{\to} (\omega)^{\mathord{<}\omega}_\lambda$, but this is harder.}
We will use this to show that all finite intervals in $H$ are good.
 
\begin{claim}
 The sequence $\langle \alpha_i,\ldots,\alpha_{n+i} \rangle$ is good for all $i,n < \omega$.
\end{claim}
\begin{proof}
 By induction on $n$. The case $n = 0$ is trivial. For the induction step, let $n \ge 1$. It suffices to show that if $\langle \alpha_i,\ldots,\alpha_{n+i-1} \rangle$ and $\langle \alpha_{i+1},\ldots,\alpha_{n+i}\rangle$ are good and the shift-invariance property \eqref{eqn:weak-homog} holds, then $\langle \alpha_i,\ldots,\alpha_{n+i} \rangle$ is good. We will show this using Figure \ref{fig:figure1} in the case $i = 0$ and $n = 3$.  (The general case is similar.)
   
 Goodness of $\langle \alpha_0,\alpha_1,\alpha_{2} \rangle$ means that all rules are followed in the upper-left $2\times 2$ subtriangle generated by $\Sigma_{\alpha_0\alpha_1}$ and $\Sigma_{\alpha_1\alpha_2}$, and goodness of $\langle \alpha_{1},\alpha_2,\alpha_{3}\rangle$ means that all rules are followed in the lower-left $2\times 2$ subtriangle  generated by $\Sigma_{\alpha_1\alpha_2}$ and $\Sigma_{\alpha_2\alpha_3}$. Then the shift-invariance property $f(\langle \alpha_0,\alpha_1,\alpha_{2} \rangle) = f(\langle \alpha_{1},\alpha_2,\alpha_{3}\rangle)$ means that $\ell_{\mathcal{T}_{\alpha_0}}(\langle x^0_0,x^0_1\rangle) = \ell_{\mathcal{T}_{\alpha_1}}(\langle x^1_0,x^1_1\rangle)$, so the label-matching rule for player II in the game $G(\mathcal{T}_{\alpha_0},\mathcal{T}_{\alpha_1})$ is satisfied when we copy the move $x^1_1$ along the dashed arrow. Finally the top-right element $x^0_2$ is given by a winning strategy for player I, so all rules are followed in the $3\times 3$ triangle and $\langle \alpha_0, \alpha_{1},\alpha_2,\alpha_{3}\rangle$ is good.
\end{proof}

The claim implies that when we play the strategies $\Sigma_{\alpha_i\alpha_{i+1}}$ for all $i < \omega$ against each other, infinitely extending Figure \ref{fig:figure1} as shown in Figure \ref{fig:figure2}, all rules are followed forever. This counts as a win for player II in each game $G(\mathcal{T}_{\alpha_i}, \mathcal{T}_{\alpha_{i+1}})$, contradicting our choice of $\Sigma_{\alpha_i\alpha_{i+1}}$ as a winning strategy for player I and completing the proof of the theorem. \qedhere
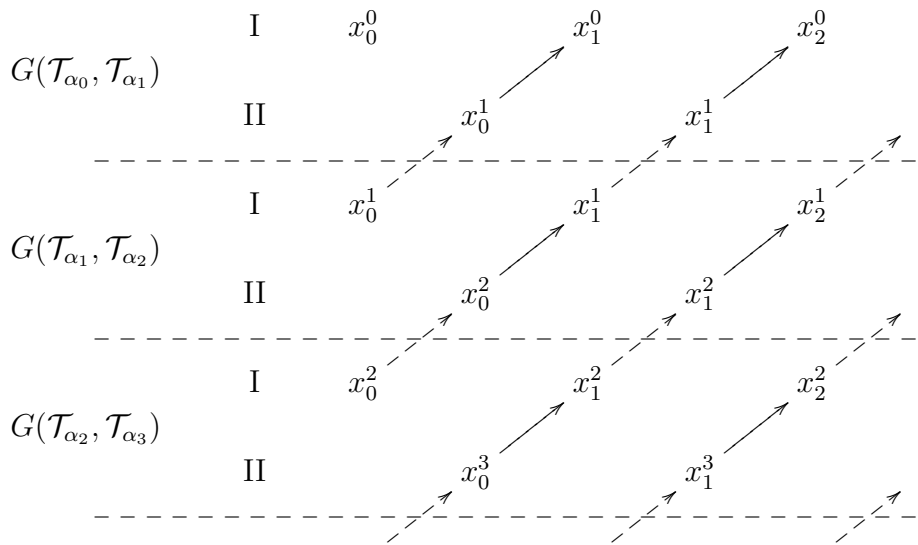
\begin{figure}[h]
 \begin{displaymath}
  \xymatrix@R-2.2pc{
                     & \text{I}  & x^0_0               &               & x^0_1               &               & x^0_2 & \\
   G(\mathcal{T}_{\alpha_0}, \mathcal{T}_{\alpha_1}) &&&               &                     &               &       & \\                                             
                     & \text{II} &                     & x^1_0\ar[uur] &                     & x^1_1\ar[uur] &       & \mbox{\phantom{$x^1_2$}}\\
   \mbox{\vphantom{$G(\mathcal{T}_{\alpha_0})$}} \ar@{--}[rrrrrrr] &&& &                     &               &       & \\
                     & \text{I}  & x^1_0\ar@{-->}[uur] &               & x^1_1\ar@{-->}[uur] &               & x^1_2\ar@{-->}[uur]      & \\
   G(\mathcal{T}_{\alpha_1}, \mathcal{T}_{\alpha_2}) &&&               &                     &               &       & \\                                               
                     & \text{II} &                     & x^2_0\ar[uur] &                     & x^2_1\ar[uur] &       & \mbox{\phantom{$x^2_2$}}\\
   \mbox{\vphantom{$G(\mathcal{T}_{\alpha_0})$}} \ar@{--}[rrrrrrr] &&& &                     &               &       & \\
                     & \text{I}  & x^2_0\ar@{-->}[uur] &               & x^2_1\ar@{-->}[uur] &               & x^2_2\ar@{-->}[uur]      & \\
   G(\mathcal{T}_{\alpha_2}, \mathcal{T}_{\alpha_3}) &&&               &                     &               &       & \\                                               
                     & \text{II} &                     & x^3_0\ar[uur] &                     & x^3_1\ar[uur] &       & \mbox{\phantom{$x^3_2$}}\\
   \mbox{\vphantom{$G(\mathcal{T}_{\alpha_0})$}} \ar@{--}[rrrrrrr] &&& &                     &               &       & \\
                     & \mbox{\phantom{\text{I}}}  & \mbox{\phantom{$x^3_0$}}\ar@{-->}[uur] &               & \mbox{\phantom{$x^3_1$}}\ar@{-->}[uur] &               & \mbox{\phantom{$x^3_2$}}\ar@{-->}[uur]      & \\                                             
  }
 \end{displaymath}
 \caption{Combining infinitely many strategies for player I}
 \label{fig:figure2}
\end{figure}

\bibliographystyle{plain}
\bibliography{labeled-trees-arxiv}

\begin{thebibliography}{1}

\bibitem{eklof1999absolutely}
Paul~C. Eklof and Saharon Shelah.
\newblock Absolutely rigid systems and absolutely indecomposable groups.
\newblock In {\em Abelian groups and modules}, pages 257--268. Springer, 1999.

\bibitem{herden2012upper}
Daniel Herden.
\newblock Upper cardinal bounds for absolute structures.
\newblock In {\em Groups and Model Theory: In Honor of R{\"u}diger G{\"o}bel's
  70th Birthday, May 30--June 3, 2011, Conference Center ``Die Wolfsburg,''
  M{\"u}lheim an Der Ruhr, Germany}, volume 576, page 137. American
  Mathematical Soc., 2012.

\bibitem{kechris1978notes}
Alexander~S. Kechris and Yiannis~N. Moschovakis.
\newblock Notes on the theory of scales.
\newblock In {\em Cabal Seminar 76--77}, pages 1--53. Springer, 1978.

\bibitem{moschovakis2009descriptive}
Yiannis~N. Moschovakis.
\newblock {\em Descriptive set theory}.
\newblock Number 155. American Mathematical Soc., 2009.

\bibitem{shelah1982better}
Saharon Shelah.
\newblock Better quasi-orders for uncountable cardinals.
\newblock {\em Israel Journal of Mathematics}, 42(3):177--226, 1982.

\bibitem{silver1970large}
Jack~H. Silver.
\newblock A large cardinal in the constructible universe.
\newblock {\em Fund. Math}, 69:93--100, 1970.

\end{thebibliography}

\end{document}